\documentclass[12pt,a4paper]{article}
\usepackage{amssymb}
\usepackage{amsmath,amsfonts,amsthm}

\usepackage{graphicx}
\usepackage{amsmath}

\usepackage{amsfonts}
\usepackage{amsthm,amscd}
\usepackage{graphicx}
\usepackage{amsmath}
\usepackage{amssymb}
\usepackage{amstext}

\newtheorem{theorem}{Theorem}
\newtheorem{corollary}[theorem]{Corollary}
\newtheorem{definition}[theorem]{Definition}
\newtheorem{example}[theorem]{Example}

\newtheorem{proposition}[theorem]{Proposition}
\newtheorem{remark}[theorem]{Remark}

\title{Thermodynamic Formalism for Haar systems in Noncommutative Integration: transverse functions and entropy of  transverse measures}

\author{Artur O. Lopes  and Jairo K. Mengue}

\begin{document}

\maketitle

\begin{abstract}

We consider here  a class of groupoids obtained via an equivalence relation (the subgroupoids of pair groupoids).   We generalize to Haar Systems in these groupoids some results related to entropy and pressure which are well known in Thermodynamic Formalism.
 We introduce a transfer operator, where the equivalence relation (which defines the groupoid) plays the role of the dynamics and  the corresponding transverse function plays the role of the {\it a priori} probability.
We also introduce the concept of invariant transverse probability  and  of entropy for an invariant transverse probability, as well as of pressure for transverse functions.
Moreover, we explore the relation between quasi-invariant probabilities and transverse measures. 
Our results are on measurable category. 
\end{abstract}

\section{Introduction}

Our purpose here is to extend the concepts of invariant probability,  entropy and pressure from Thermodynamic Formalism to the setting of quasi-invariant probabilities, transverse functions and transverse measures, which are naturally defined on groupoids and Haar systems.  The groupoids we consider here  will always be obtained via an equivalence relation (also called subgroupoids of the pair groupoid - see section 3 in \cite{wei} for definitions). Most of our results
are on the measurable category.

The results we obtain can be seen as similar to the classical results of Thermodynamic Formalism. We refer the reader to \cite{PP} and \cite{Walters} for results on Thermodynamic Formalism and to \cite{Con}, \cite{K}, \cite{Rena1} for results on Haar systems, groupoids and transverse measures (see also \cite{CLM} and \cite{BK}  for a strictly   measure theoretical perspective). But, in any case we point out that the present work is {\bf self contained} for readers familiar with Thermodynamic Formalism.

The classical Kolmogorov-Sinai entropy is defined for probabilities which are invariant for a deterministic dynamical system. We point out that for a Haar system on a groupoid there is  (in general) no underlying dynamical system. To realize that entropy depends on the {\it a priori} probability (as described in [24]) is the key issue for finding a  suitable procedure  to extend this formalism (of thermodynamic formalism for H\"{o}lder potentials) to Haar systems. When the alphabet is not countable (the so called generalized $X Y$ models) the definition of entropy via dynamical partitions is not suitable anymore and an {\it a priori} probability is necessary.

In the dictionary to be used here the transverse function of a Haar system is the mathematical object corresponding to the {\it a priori} probability and the equivalence relation (the groupoid) plays the role of the dynamics. The role of the potential is played by the  modular function and, finally, the transverse measures and quasi-invariant probabilities in Haar systems play the role of the measures in thermodynamic formalism.

%The paper \cite{ACR} complement \cite{LMMS} showing unicity of the equilibrium probability and the paper \cite{CME} shows that the equilibrium state has full support (all this for Lipschitz potentials). THIS WAS RELOCATED FOR SECTION 3. 

%The concept of transverse measure - acting on measurable transverse functions - is a natural generalization of the classical concept of measure - acting on  measurable functions -  (see remark \ref{re1} after Theorem \ref{density} and also example \ref{so}).

In section \ref{int} we introduce the main notations and definitions concerning Haar Systems, which includes the concepts of transverse function, modular function, transverse measure and quasi-invariant probability. 

Theorem 54 in \cite{CLM} shows that DLR probabilities (see \cite{CL2} for definition) are quasi-invariant probabilities for a certain class of  H\"{o}lder modular functions in the case the alphabet is finite. In section 3 here we show an analogous results for the case the alphabet is a compact metric space. We consider as an example the generalized XY model, as studied in \cite{LMMS}, and we show that any eigenprobability for the dual Ruelle operator is a quasi-invariant probability for the associated Haar System.  
We assume that the modular function is just of H\"{o}lder class on this section.

The results on the next sections will be on the measurable category.

In section \ref{nonc} we consider particular modular functions and develop the main part of the paper studying Haar Systems from a Thermodynamic Formalism point of view. We introduce the concepts of Haar-invariant probabilities and Haar-invariant transverse probabilities and of entropy for Haar-invariant transverse probabilities and pressure for transverse functions. In  \cite{Con} it is presented the relation of transverse measures and quasi-invariant probabilities (see also section 5 in \cite{CLM}). In section \ref{trans} we prove an equivalence between the Haar-invariant probabilities and Haar-invariant transverse probabilities.

 In  section \ref{examples} we exhibit examples and analyze the relations between the concepts introduced in this work with the classical ones for Thermodynamic Formalism.
In Section \ref{haar-dy}, which considers a large class of dynamically defined groupoids
and quasi-invariant probabilities, the Rokhlin's disintegration theorem plays
an important role.

 We refer to \cite{Weiss}, \cite{Nad} and \cite{Schmi} for classical results on measurable dynamics. \cite{K}, \cite{Fel1} and \cite{Fel2} are the classical references for Haar systems when the transverse function is the counting measure.
For the relation of quasi-invariant probabilities with KMS states of $C^*$-algebras (and von Neumann algebras) see \cite{Rena1}, \cite{Rena2}, \cite{GL}, \cite{Cas1}, \cite{HR}, \cite{Rue}, \cite{KR}, \cite{Ana1}, \cite{Ana2}, \cite{Sakai} \cite{LM} and \cite{CLM}.
A different kind of relation between KMS states of $C^*$-algebras  and Thermodynamic Formalism is described in  \cite{Exel1}, \cite{EL1} and \cite{EL2}.
 We refer the reader to \cite{CLM} for an extensive presentation  of Haar systems and non commutative integration on groupoids obtained via an equivalence relation (some results are for dynamically defined equivalence relations).

\section{Transverse functions and transverse measures} \label{int}

Consider a metric space $\Omega$ with metric $d$
and denote by $\mathcal{B}$ the Borel sigma-algebra on $\Omega$. We fix an equivalence relation $R$ on $\Omega$ and if two points $x,y$ are related, we write
$x\sim y$. We denote by $G\subset \Omega \times \Omega$ the associated {\bf  Groupoid}
$$ G =\{(x,y)\in\Omega \times \Omega \,|\, x\sim y\}$$
and by $[x]=\{y\in \Omega \,|\,x\sim y\}$ the class of $x$.

This corresponds to subgroupoids of the pair groupoid (see section 3 in \cite{wei}).  These are the only kind of groupoids we will consider here.

 Extreme examples of such groupoids are the cartesian product (pair groupoid)  $G=\Omega\times \Omega$ when $[x]=\Omega\,,\, \forall \,\,x$ (where any two points are related) and the diagonal $G=\{(x,x)\,|\,x\in \Omega\}$ when $ [x]=\{x\}\,,\,\forall\,\,x$ (where each point is related just with itself).

We consider over $G$ the topology induced by the  product topology on $\Omega \times \Omega$ and denote also by
$\mathcal{B}$  the Borel sigma-algebra induced  on $G$.

\begin{definition} We say that $G$ is a {\bf measurable groupoid} if the maps
$$s(x,y)=x, \,\,\,r(x,y)=y,\,\,\, h(x,y)=(y,x)\,\,\,\text{and}\,\,\,Z(\,(x,s),(s,y)\,)= (x,y),$$
are  Borel measurable.
\end{definition}

If $G$ is a {measurable groupoid}, then, in particular, each class $[x],\, x\in \Omega$ is a measurable set of $\Omega$. In all this work we suppose that $G$ is a measurable groupoid obtained from a general equivalence relation.

\begin{remark} \label{alg}  On the general definition of groupoids (see  page 100 of  \cite{K}) appears the  concept of a set of morphisms $\gamma: x \to y$, for each pair of objects $x$ and $y$. In our work the objects are the points of $\Omega$ and given two points $a\sim b$ in $\Omega$ there exist a unique morphism  $\gamma:a\to b$ which is represented by $(a,b)$. It follows that $s(\gamma)$ and $r(\gamma)$ in \cite{K} just correspond to the projections $s(x,y)= x$ and $r(x,y)=y$. The morphisms are not explicitly used in our work. 
\end{remark}

\begin{definition}\label{kernel}
A {\bf kernel}   $\lambda$ on the measurable groupoid $G$ is a map  of $\Omega$ in the space of measures on  $(\Omega,\mathcal{B})$, such
that,

1)  $\forall y \in \Omega$, the measure $\lambda^y$ has support on $[y]$,

2) $\forall A \in \mathcal{B}$, we have that $\lambda^y (A)$, as a function of $y$, is measurable.

\end{definition}

In some sense, the two above items corresponds to items i) and ii) in definition 5.14 (disintegration of a measure with respect to a partition) in \cite{OV}. See also Theorem \ref{Rokhlin} below.

There is a subtle point on item 1) in  the definition of kernel. An alternative definition could be:\,
1) for any $y\in \Omega$ we have that $\lambda^y (\Omega -[y])=0.$ Some of the results we get here could be obtained with this alternative condition (but we will not elaborate on that).

\begin{definition} \label{krt}
A {\bf transverse function } $\nu$ on the measurable groupoid $G$ is a kernel satisfying
$\nu^x=\nu^y$, for any $(x, y) \in G$ (that is $x\sim y$).
We denote by $\mathcal{E}^+$ {\bf the set of transverse functions}.
\end{definition}

The concept of transverse function is a natural generalization of the concept of measurable non negative function $f: \Omega \to \mathbb{R}$ (see remark \ref{re1} after Theorem \ref{density}).

We denote by $\mathcal{E}$ the set of signed transverse functions. More precisely, $\nu \in \mathcal{E}$, if the family of measures $(\nu^{y})^+$ and $(\nu^{y})^-$ that form, for each $y$, the Hann-Jordan decomposition of $\nu^y$ are both transverse functions. An important example of signed transverse function is $\mu^{y}(dx):=f(x)\nu^{y}(dx)$ where $f:\Omega\to \mathbb{R}$ is measurable and bounded and $\nu \in \mathcal{E}^+$ satisfies $\int 1\,\nu^y(dx)=1\,,\,\forall y$.

\begin{definition} \label{Haa} The pair $(G,\hat{\nu})$, where $G$ is a measurable groupoid and $\hat{\nu}$ a transverse function will be called a {\bf Haar system}.
\end{definition}

\begin{example} \label{Bak}
Take $\Omega=[0,1] \times[0,1]$ and consider the groupoid $G$ defined from the equivalence relation
	$x=(a_1,a_2) \sim y =(b_1,b_2)$, if $a_1=b_1$. The classes can be identified as vertical lines of the unitary square. They are the local unstable leaves of a Baker map (see \cite{CLM} for a complete discussion).

	 Given a probability $\nu$ on $[0,1]$ and a measurable function $\varphi:\Omega \to [0,+\infty)$, we can interpret $\varphi$ as a family of density functions $\varphi^{a_1}:[\, a_1,\,\cdot \,]\to \mathbb{R}$, each one acting in a vertical fiber, and define a transverse function $\hat{\nu}$ which coincides with $\varphi^{a_1}d\nu$ in the fiber $[\, a_1,\,\cdot \,]$. Then $(G,\hat{\nu})$ is a Haar system.

\end{example}

A kernel $\lambda$ is characterized by the operator
\[\lambda(f)(y) = \int f(x,y)\,\lambda^{y}(dx), \]
acting over $\lambda$-integrable functions $f:G\to \mathbb{R}$.
Given a kernel $\lambda$ and a $\lambda-$integrable function $g\geq 0$  we denote by $g\lambda$ the kernel $(g\lambda)^{y}(dx)=g(x,y)\lambda^{y}(dx)$. In this way
\[(g\lambda)(f)(y)  = \int f(x,y)g(x,y)\lambda^y(dx).\]

The convolution of two kernels $\lambda_1$ and $\lambda_2$ is the kernel $\lambda_1*\lambda_2$ satisfying
\begin{equation}\label{convolution}
(\lambda_1*\lambda_2)(f)(y) = \int f(x,y)\,(\lambda_1*\lambda_2)^{y}(dx) = \iint f(s,y)\lambda_2^{x}(ds)\lambda_1^{y}(dx),
\end{equation}
for any $\lambda_2^{x}(ds)\lambda_1^{y}(dx)$-integrable  function $f$.

\begin{definition} \label{Har} A {\bf modular function} over the groupoid $G$ is a measurable function $\delta:G \to \mathbb{R}$,
such that, for any $x\in \Omega$, and any pair $y,z \in [x]$ we have that $\delta(x,y)\, \delta (y,z)= \delta(x,z).$

\end{definition}

\begin{definition} \label{coco}
  A {\bf transverse measure} $\Lambda$ for the groupoid $G$ and the modular function  $\delta$ is a linear\footnote{which means, $\Lambda(a\nu+b\mu)=a\Lambda(\nu)+b\Lambda(\mu)$ for any $\mu,\nu \in \mathcal{E}^+$ and $a,b \in \mathbb{R}$ such that $(a \nu+b\mu) \in \mathcal{E}^{+}$} function  $\Lambda :\mathcal{E}^+\to \mathbb{R}^+$, which satisfies the property:  for each kernel $\lambda$ such that  for any $y$,  we have $\lambda^y(1) =1$, then, if $\nu_1$ and $\nu_2$ are transverse functions satisfying $\nu_1 \,*\, (\delta \lambda)=\nu_2$,  it will be required  that
\begin{equation}
\label{gogo}  \Lambda (\nu_1)= \Lambda (\nu_2).
\end{equation}

\end{definition}

\medskip

The action of a transverse measure $\Lambda$  on $\mathcal{E}^+$ can be linearly extended to $\mathcal{E}$ (separating in the negative and positive part).

\medskip

As we will see later the concept of transverse measure (acting on transverse functions) is a natural generalization of the classical concept of measure (acting by integration on functions) for the setting of Haar Systems (see remark \ref{re1} after Theorem \ref{density} and also example \ref{so}).

\medskip

\begin{definition} \label{xix}
Given a modular function $\delta$, a grupoid $G$ and a fixed transverse function $\hat{\nu}$, which is a probability for any $y$, we say that a probability $M$ on $\Omega$ is {\bf quasi-invariant} for the Haar system $(G,\hat{\nu})$ if for any bounded measurable function $f:G\to\mathbb{R}$,
$$\iint f(y,x)\, \hat{\nu}^y (dx) d M( y) =\iint f(x,y)  \delta(x,y)^{-1} \hat{\nu}^y (dx) d M(y).$$

\end{definition}

\medskip

In the Theorem \ref{equivalent} we exhibit, under certain hypotheses, a relation between transverse measures and quasi-invariant probabilities for modular functions in the particular form $\delta(x,y)=e^{V(y)-V(x)}$.

There are different (analogous) definitions of quasi-invariant probability. For example in \cite{Mi1} and \cite{Mi2} there is no mention to transverse functions and the concept is defined via Borel injections (it is considered the concept of $\delta$-invariant probability). For the existence of quasi-invariant probabilities in the measurable dynamics setting see
appendix of \cite{Rena3} or \cite{Kri1}, \cite{Kri2}, \cite{Schmi} and \cite{Nad}.

An interesting class of groupoids are described by Definition 1.9 and 1.10 in \cite{LO}.  The authors called continuous (or, Lipschitz)   groupoid, a groupoid defined by an equivalence relation on the symbolic space $X=\{1,2,...,d\}^{\mathbb{N}}$, where given two close elements $x,y \in X$, there is a continuous (or, Lipschitz) correspondence such that one  can associate elements on each of the finite classes $[x]$ and $[y]$ (which have same cardinality). In this case a kind of Ruelle operator
(the Haar-Ruelle operator) can be defined and stronger properties  (compare to the measurable setting we consider here) can be obtained.

Here the transverse function $\hat{\nu}$ (defining a Haar system) plays an important role. Note, however, that in  the definition of transverse measure it is not mentioned a fixed Haar system.

Given an equivalence relation defining a groupoid $G$, suppose that $[x]$ is a finite set for any $x\in \Omega$. The saturation of a measurable set $B\subset \Omega $ is the subset of $G$ given by
$$ S [B] = \cup_{x\in B}  [x].$$
Consider the Haar system where the transverse function is the counting measure. In section 4.2 of \cite{BK} (or, in \cite{Fel1}) it is shown (a classical result) that a probability $M$ is quasi-invariant for some modular measurable function  $\delta$, if and only if, satisfies the condition
\begin{equation} \label{leo} M(B)=0 \,\,\text{implies that}\, M(S[B])=0, \,\,\text{for all Borel sets}\,B \in \Omega.\end{equation}

\begin{example}

Consider the example of Haar system where $\Omega=[0,1] \times[0,1]$, each class is a vertical line and $\hat{\nu}^{a}$ is the Lebesgue measure on each line. The classes are the local unstable leaves of a nonlinear Baker map  (see \cite{CLM} for a precise definition) $F:\Omega \to \Omega$ given by
$$F(a_1,a_2) =(H (a_1,a_2), T (a_2)),$$
where $T:[0,1]\to [0,1]$ is a $C^{1 + \alpha}$ expanding transformation
($F$ is a simplified version of an Anosov transformation). There is an interesting relation of the SBR probability for $F$
and the quasi-invariant probability $M$ (see \cite{CLM})  for the modular function $\delta$ given by

$$\delta(\,(a_1,y_1)\,,\, (a_1,y_2)\,)=\frac{V(a_1,y_1)}{V(a_1,y_2)}= \Pi_{n=1}^\infty\,\, \,\frac{T ' ( b^n (a_1,y_1))}{  T ' ( b^n (a_1,y_2))},$$
where for each $n\geq 0$, the points $b^n (a_1,y_1)$ and  $b^n (a_1,y_2))$ are, respectively, the successive
$n$-preimages,  of $(a_1,y_1)$ and  $(a_1,y_2)$ which are close by in the same
vertical line (local unstable).

\end{example}

\section{The inspiring model} \label{insp}

The purpose of this section is to present a preliminary example which can help the reader to understand why is natural the reasoning we will pursue on the following sections.

On this section is fixed a compact metric space $K$ and the associated Bernoulli space\footnote{It describes
	the Statistical Mechanics system  where  the fiber of spins is the metric space $K$ (that can be finite or not) and each site of the lattice is on  $\mathbb{N}$.} \,\,$\Omega = K^\mathbb{N}$.
Points $x\in \Omega$ are denoted by $x=(x_1,x_2,...,x_n,...).$ This is called the generalized XY model studied in \cite{LMMS} (see also \cite{ACR} and \cite{CME}).

The groupoid $G \subset \Omega \times \Omega$ is defined from the equivalence relation $x=(x_1,x_2,x_3,...) \sim y=(y_1,y_2,y_3,...)$, if $x_j=y_j$, for all $j\geq 2.$ Observe that $x\sim y$ means that $x=(a,y_2,y_3,...)$, for some $a\in K$, which is equivalent to $\sigma(x)=\sigma(y)$, where $\sigma$ is the shift map.
 We say that a groupoid $G$ which was defined in  such way was {\bf  dynamically defined}. We will consider a more general class of dynamically defined groupoids in section \ref{haar-dy}.

We denote by $m$ a fixed  {\it a priori} probability on $K$ (with support equal to $K$).
Consider the transverse function $\hat{\nu}$, such that, for each $y\in \Omega$ and continuous function $f:[y]\to\mathbb{R}$,
we have

 $$\int f(x) \hat{\nu}^y(dx) = \int  f(a,y_2,y_3,...) \, d m(a) .$$

In the case $K=\{1,2,...,d\}$ is natural to take the probability $m$  such that  each point in $K$  has $m$-mass equal to $1/d$, but, ``by no means'' this has to be the only choice. Similarly, when $K=S^1$ it is also natural to consider the Lebesgue probability $d a$ as the {\it a priori} probability (see \cite{BCLMS}).

\medskip

Given a H\"{o}lder function $V: \Omega \to \mathbb{R}$, the associated Ruelle operator (acting on continuous functions) is defined by
$$ f \to \, \mathcal{L}_V (f) (x)= \int e^{ V(a,x_1,x_2,...)} f(a,x_1,x_2,...) \, d m(a).$$

Consider  a H\"{o}lder function $V: \Omega \to \mathbb{R}$ and take as modular function $\delta(x,y)= e^{V(y)- V(x)}.$  Denote by $c$ the eigenvalue and $\varphi$ the eigenfunction for the transfer (Ruelle) operator $\mathcal{L}_V$. We denote by $\rho$ the eigenprobability for $\mathcal{L}_V^*,$ which satisfies $L_V^{*}(\rho)=c\rho$. In this way, if $\int \varphi \, d\rho = 1$, the probability $\mu:=\varphi\, \rho$ is the equilibrium probability for $V$ (see \cite{PP}). We denote by $U$ the normalized H\"{o}lder potential
$$U= V + \log \varphi - \log (\varphi \circ \sigma) - \log (c).$$

Given  $k_0\in K$, consider fixed  the point $z_0=(k_0)^\infty=(k_0,k_0,...) \in \Omega$.
As, for any continuous function $g: \Omega \to \mathbb{R}$,
$$\int g(x) d \mu(x) = \lim_{n \to \infty} \mathcal{L}_U^n (g)  \,(z_0)  ,$$
then, for any continuous function $h$,
$$\int h(x) d \rho(x) = \int \frac{h(x)}{\varphi(x)} d \mu(x) = \lim_{n \to \infty} \mathcal{L}_U^n \left(\frac{h}{\varphi}\right)  \,(z_0)   =  \lim_{n \to \infty} \frac{1}{c^n \varphi
(z_0)}\,\, \mathcal{L}^n_V (h)\,(z_0). $$
Consequently
$$\int h(x) d \rho(x)=        \lim_{n \to \infty} \frac{\mathcal{L}_V^n (h)  \,(z_0)}{\mathcal{L}_V^n (1)  \,(z_0) } .$$

Such kind of expression appears in \cite{CL2}. The next result is a generalization of a similar one in section 4 in \cite{CLM}.

\begin{proposition} \label{mimeo}
Under above hypotheses and notations, the eigenprobability $\rho$ for the dual Ruelle
operator $\mathcal{L}_V ^*$ is quasi-invariant for the modular function $\delta(x,y)= e^{V(y)- V(x)}$, that is, for all continuous function $f:G\to\mathbb{R}$ we have
$$\iint f(y,x)\, \hat{\nu}^y (dx) d \rho( y) =\iint f(x,y)  \delta(x,y)^{-1} \hat{\nu}^y (dx) d \rho(y).$$
\end{proposition}

\begin{proof}
In this proof we denote $dm(a_1)dm(a_2)...dm(a_n)$ by $dm(a_1,...,a_n)$.  We write $y=(y_1,y_2,y_3,...)$ and for $x\in [y]$ we write $x= (a,y_2,y_3,...)$.
Let us define two auxiliary functions
\[g_1(y) :=\int f((a,y_2,y_3,...),(y_1,y_2,y_3,...) e^{V(a,y_2,y_3,...)} dm(a),\]
and
\[g_2(y):=\int f((y_1,y_2,y_3,...),(a,y_2,y_3,...) dm(a). \]

Then,
\[\iint f(x,y)  \delta(x,y)^{-1} \hat{\nu}^y (dx) d \rho(y)= \iint f(x,y) e^{V(x)- V(y)}  \hat{\nu}^y (dx) d \rho(y)\]
\[=\iint f((a,y_2,y_3,...),(y_1,y_2,y_3,...) e^{V(a,y_2,y_3,...)} dm(a)e^{- V(y_1,y_2,y_3,...)}  d \rho(y)  \]
\[= \int g_1(y)e^{-V(y)}\,d\rho(y)  =  \lim_{n \to \infty} \frac{\mathcal{L}_V^n (g_1\cdot e^{-V})  \,(z_0)}{\mathcal{L}_V^n (1)  \,(z_0) }\]
\footnotesize
\[= \lim_{n \to \infty} \frac{\int e^{S_n(V)(a_1,...,a_n,z_0)}g_1(a_1,...,a_n,z_0)e^{-V(a_1,...,a_n,z_0)} \,dm(a_1,...,a_n) }{\mathcal{L}_V^n (1)  \,(z_0) }\]
\[=  \lim_{n \to \infty} \frac{\int e^{S_{n-1}(V)(a_2,...,a_n,z_0)}g_1(a_1,...,a_n,z_0) \,dm(a_1,...,a_n) }{\mathcal{L}_V^n (1)  \,(z_0) }\]
\[=  \lim_{n \to \infty} \frac{\iint e^{S_{n}(V)(a,a_2,...,a_n,z_0)}f((a,a_2,...,a_n,z_0),(a_1,...,a_n,z_0)) \,dm(a)dm(a_1,...,a_n) }{\mathcal{L}_V^n (1)  \,(z_0) }\]
\[=  \lim_{n \to \infty} \frac{\iint e^{S_{n}(V)(a_1,a_2,...,a_n,z_0)}f((a_1,a_2,...,a_n,z_0),(a,a_2,...,a_n,z_0)) \,dm(a)dm(a_1,...,a_n) }{\mathcal{L}_V^n (1)  \,(z_0) }\]
\[=  \lim_{n \to \infty} \frac{\int e^{S_{n}(V)(a_1,a_2,...,a_n,z_0)}g_2(a_1,a_2,...,a_n,z_0) \,dm(a_1,...,a_n) }{\mathcal{L}_V^n (1)  \,(z_0) }\]
\normalsize
\[=\int g_2(y)\,d\rho(y)=\iint f((y_1,y_2,y_3,...),(a,y_2,y_3...)dm(a)      d\rho(y)\]
\[=\int  f(y,x) \hat{\nu}^{y}(dx) d\rho(y).\]

\end{proof}

\medskip

We point out that the above probability $\rho$ is not the unique quasi-stationary probability for such $\delta$ (see end of section 4 in \cite{CLM}).
\medskip

\section{A Thermodynamic Formalism point of view for Haar Systems} \label{nonc}

Now we return to the analysis of general Haar Systems (not necessarily as the previous generalized XY model).
We consider a metric space $\Omega$ with the Borel sigma-algebra $\mathcal{B}$ and a measurable groupoid $G$. In all this section is fixed a Haar-system $(G,\hat{\nu})$ where the transverse function $\hat{\nu}$ (see definition \ref{krt}) satisfies $\int \,1\, \hat{\nu}^y (dx) =1, \,\,\forall y$.

In the present setting the dynamical action  is replaced by the equivalence relation which is described by the groupoid $G$.
The  transverse function $\hat{\nu}$  will play here the role of the {\it a priori} probability in the Thermodynamic Formalism for the generalized $ X Y$ model.

\medskip
\subsection{A transfer operator for Haar systems} \label{prem}

We will consider modular functions in the form $\delta(x,y) = e^{V(y)-V(x)}$, where $V: \Omega \to \mathbb{R}$ is a bounded and {\bf measurable function} . Then a probability $M$ on  $\Omega$ is quasi-invariant  for the Haar system $(G,\hat{\nu})$ and $V$, if  satisfies the property
$$\iint f(y,x)\, \hat{\nu}^y (dx) d M( y) =\iint f(x,y)  e^{V(x)-V(y)} \hat{\nu}^y (dx) d M(y),$$ for all measurable and bounded function $f$ (see definition \ref{xix}).

\begin{definition}\label{normalized} A bounded and measurable function $V:\Omega \to\mathbb{R}$ is \textbf{Haar-normalized} for the Haar system $(G,\hat{\nu})$ (or simply, $\hat{\nu}$-normalized) if it satisfies
\[ \int e^{V(x)}\, \hat{\nu}^{y}(dx) = 1 , \,\,\,\forall\, y \in \Omega. \]
\end{definition}

The above property corresponds in classical Thermodynamic Formalism to the concept of normalized potential for the Ruelle operator. Note that we do not assume that $V$ is of H\"{o}lder class.

\begin{definition} \label{refs}  A {\bf Haar-invariant probability} for the { Haar system $(G,\hat{\nu})$} will be  a probability $M$ on $\Omega$, such that, for some Haar-normalized function $V:\Omega \to \mathbb{R}$,
$$\iint f(y,x)\, \hat{\nu}^y (dx) d M( y) =\iint f(x,y)  e^{V(x)-V(y)} \hat{\nu}^y (dx) d M(y),$$
for all measurable bounded function $f$.
\end{definition}

\begin{remark} \label{rur}        Any Haar-invariant measure is  quasi-invariant.
\end{remark}

\begin{remark}\label{rur1}
In the Proposition  \ref{mimeo} the probability $\rho$, which is an eigenprobability for $\mathcal{L}_V^*$, is also quasi-invariant. It is necessary to assume that $V$ is a normalized potential (for the Ruelle operator) in order to exhibit the case where the probability $\rho$ is an invariant measure for the shift map. In this case such normalized potential $V$ is also Haar-normalized and $\rho$ is also a {Haar-invariant measure} for the Haar system $(G,\hat{\nu})$.
\end{remark}

\bigskip

   A probability $M$ is Haar-invariant and associated to the normalized function $V$, iff, for any test function $f$ we have
\begin{equation}
\label{eq3}
\iint f(y,x)e^{V(x)}\,\hat{\nu}^{y}(dx)dM(y) = \iint f(x,y)e^{V(x)}\,\hat{\nu}^{y}(dx)dM(y).
\end{equation}

Furthermore, if $M$ is Haar-invariant and associated to the normalized function $V$ then, considering the particular case where $f(z,w)=f(w)$, we get from (\ref{eq3}) that
\begin{equation}\label{invariant}
 \iint f(x) e^{V(x)}\hat{\nu}^{y}(dx)dM(y)=\int f(y) dM(y) .
 \end{equation}
It follows that $M$ is a fixed point for the operator $H_V^{*}$, defined below.

\begin{definition}
Given a measurable and bounded function $U:\Omega \to \mathbb{R}$ we define the operator $H_U$ acting in measurable and bounded functions by
\begin{equation} \label{transR} H_U(f)(y) = \int e^{U(x)}f(x) \hat{\nu}^{y}(dx).
\end{equation}
If $V$ is Haar-normalized, the dual operator $H_V^{*}$ restrict to the convex set of probabilities on $\Omega$ satisfies, for any measurable and bounded function $f$,
\begin{equation}\label{dual}
 \int f \, dH_V^{*}(M_1) := \int e^{V(x)} \, f (x) \,  \hat{\nu}^{y} (dx) d M_1 (y) = \int H_V(f) dM_1.
\end{equation}
\end{definition}

\bigskip

The above operator $H_V$ {\bf is not the Ruelle operator }  when one considers the particular setting of section \ref{insp}. If $\mathcal{L}_V$ is the Ruelle operator, then,
$$ H_V(f) (y)= \mathcal{L}_V (f) (\sigma(y)),$$
where $\sigma$ is the shift map. We remark also that $H_V$ is not the Haar-Ruelle operator studied in \cite{LO}.

\begin{proposition}\label{fixdual} $M$ is Haar-invariant for $(G,\hat{\nu})$, iff, there exists a Haar-normalized (measurable) function $V$, such that,
	$M$ is a fixed point for the  operator $H_V^{*}$ defined in (\ref{dual}).
We will call $e^V$ \textbf{a Haar-Jacobian of $M$}.

\end{proposition}

\begin{proof}
In the above was shown that any Haar-invariant measure is a fixed point for the operator $H_V^{*}$.
Now we suppose that a probability $M$ satisfies, for any measurable and bounded function $F$,
\[ \iint  e^{V(z)}F(z)\hat{\nu}^{y}(dz)dM(y)=\int F(y) dM(y), \]
where $V$ is Haar-normalized. In this way we want to prove that $M$ is Haar invariant, that is, it satisfies (\ref{eq3}).

We begin  analyzing the left hand side of (\ref{eq3}). We fix a test function $f$, and, let	$F(y):= \int f(y,x)e^{V(x)}\,\hat{\nu}^{y}(dx)$. Then,
\[ \iint f(y,x)e^{V(x)}\,\hat{\nu}^{y}(dx)dM(y) = \int F(y) dM(y)\stackrel{hypothesis}{=}\]
\[ = \iint F(z) e^{V(z)}\hat{\nu}^{y}(dz)dM(y) = \iiint f(z,x)e^{V(x)}\,\hat{\nu}^{z}(dx) e^{V(z)}\hat{\nu}^{y}(dz)dM(y) \]
\[\stackrel{\hat{\nu}^{z}=\hat{\nu}^{y}}{=} \iiint f(z,x)e^{V(x)}\,\hat{\nu}^{y}(dx) e^{V(z)}\hat{\nu}^{y}(dz)dM(y) \]
\[= \iiint f(z,x)e^{V(x)+V(z)}\,\hat{\nu}^{y}(dx) \hat{\nu}^{y}(dz)dM(y). \]
Now we apply similar computations on the right hand side of (\ref{eq3}).
\[ \iint f(x,y)e^{V(x)}\,\hat{\nu}^{y}(dx)dM(y) = \iiint f(x,z)e^{V(x)}\hat{\nu}^{z}(dx)e^{V(z)}\hat{\nu}^{y}(dz)dM(y)\]
\[=\iiint f(x,z)e^{V(x)+V(z)}\,\hat{\nu}^{y}(dx) \hat{\nu}^{y}(dz)dM(y)\]
\[\stackrel{x \leftrightarrow z}{=} \iiint f(z,x)e^{V(x)+V(z)}\,\hat{\nu}^{y}(dz) \hat{\nu}^{y}(dx)dM(y).\]
It follows from Fubini's Theorem that the two sides of (\ref{eq3}) are equal.	
\end{proof}

\medskip

In the present setting - where there is no dynamics - the above result shows us that a natural way for getting the analogous concept of  invariant measure can be obtained via a transfer operator (which is analogous to the Ruelle operator in symbolic dynamics). We observe that in this setting the operator is defined from an {\it a priori} measure that depends of the point $y$, which is the transverse function $\hat{\nu}$.

\bigskip

In the sequel on this section we describe some properties of the operators $H_U$ and $H_V^*$.

\begin{proposition}\label{eigenvalue} For any given measurable and bounded function $U$, consider the operator $H_U$ as defined in  (\ref{transR}) and the function $\tilde U (y) = \int e^{U(x)}\hat{\nu}^{y}(dx)$, which is constant on classes.
	\newline
1. If $f$ is constant on classes then \[H_U(f)(y) = \tilde U(y)f(y) .   \]
2. The function $V:=U-\log(\tilde U)$ is Haar-normalized.\newline
3. If there exists some positive eigenfunction $g$ (for a certain eigenvalue) for the operator $H_U$, then $\tilde U$ need to be constant. This constant value  $\tilde U$  is the corresponding eigenvalue (it is also positive). \newline
4. If $\tilde U$ is constant and $\lambda:=\tilde U(y)\,,\,\forall y$, then a (measurable) function $g$ is eigenfunction for $H_U$, if and only if, $g$ is constant on classes. In this case $H_U(g)=\lambda g$.
\end{proposition}

\begin{proof}-\newline
Proof of 1.
\[H_U(f)(y) =\int e^{U(x)}f(x) \hat{\nu}^{y}(dx) = \int e^{U(x)}f(y) \hat{\nu}^{y}(dx) \]\[=  [\int e^{U(x)} \hat{\nu}^{y}(dx)]f(y)=\tilde U(y)f(y). \]
Proof of 2.
\[ \int e^{U(x)-\log(\tilde U(x))} \,\hat{\nu}^y(dx) = H_U(\frac{1}{\tilde U})(y) \stackrel{item \, 1.}{=} \frac{\tilde U (y)}{\tilde U(y)} =1.    \]	
Proof of 3. Suppose that for some measurable and bounded function $g>0$ and real number $\lambda$ we have that $H_U(g)=\lambda g$. As $H_U(g)$ is constant on classes and $H_U(g)=\lambda g$, the function $g$ is necessarily constant on classes too. It follows from 1. that $H_U(g) =\tilde U g$ and therefore $\lambda g = \tilde U g$. 	As $g$ is positive we get $\tilde U = \lambda$. It follows from definition of $\tilde U$ that $\lambda> 0$. \newline
Proof of 4. If $g$ is constant on classes, then, from 1. we get that $H_U(g)=\tilde U g = \lambda g$. On the other hand, if $g$ is an eigenfunction of $H_U$, following the proof of 3., we get that $g$ is constant on classes, and, furthermore $H_U(g)=\lambda g$.

\end{proof}

 In the above result the function $\tilde U$ plays the role of the eigenvalue of the operator $H_U$.
The normalization procedure (getting a Haar-normalized $V$ from the given $U$) described by item 2. on the above proposition is much more simpler that the corresponding one in Thermodynamic Formalism (where one has to add a coboundary).

 \medskip

\begin{corollary} Suppose that $V$ is Haar-normalized. Consider the operator $H_V$ defined by $(\ref{transR})$.
	If $f$ is constant on classes, then, $H_V(f)=f$. Particularly, $H_V\circ H_V = H_V$.
\end{corollary} 	
\begin{proof}
	If $V$ is normalized, then, $\tilde V = 1$, and consequently, if $f$ is constant on classes, $H_V(f)=\tilde V f = f$. Consequently, $H_V (H_V(g)) = H_V(g)$ for any measurable and bounded function $g$, because $H_V(g)$ is constant on classes.
\end{proof}

\begin{example}
If we consider the Groupoid  defined from the equivalence relation $x \sim y$, iff, $x=y$, then we have $[y]=\{y\}$, and, therefore $\hat{\nu}$ is trivial, that is, $\hat{\nu}^y = \delta_{y}$, over the set $\{y\}$. In this case, the unique Haar-normalized function is $V\equiv 0$. Furthermore, for any function $U$ we get that $\tilde U = e^{U}$ and $U-\log(\tilde U) = 0=V$. In this model is quite simple to see that $H_V = Id$ and consequently any probability $M$ is fixed for $H_V^{*}$. Then, we have:\newline
1. The fixed probability of  $H_V^{*}$ is not unique\newline
2. If $f$ is not constant, $H_V^{n}(f):=H_V^{n-1} \circ H_V(f)=f$ does not converge to a constant (it does not converge, for instance, to any possible given $\int f \,dM$).
\end{example}

Analyzing equation (\ref{invariant}) it is  natural the following reasoning: this equation could be solved independently for each class $[y]$ and after this, the solutions could be combined (adding class by class) in order to get a probability $M$ over all the space $\Omega$. Furthermore, the weight that $M$ has on each class seems to have no relevance in order to get a Haar-invariant measure. This  remark is more formally presented  in the next theorem.

\begin{theorem}\label{Mandmu} Let $V$ be a Haar-normalized function and $\mu$ be any probability measure on $\Omega$. There exists a unique Haar-invariant probability $M$ with Jacobian $e^{V}$ and such that, for any bounded and measurable function $g$, constant on classes, we get
	\[ \int g\,dM = \int g\, d\mu.\]
\end{theorem}

\begin{proof}
Let $M= H_V^{*}(\mu)$. Then, for any integrable function $f$ we have
\[\int  f\, dM = \int H_V(f)\, d\mu. \]
Particularly, as $H_V\circ H_V =H_V$ we get, for any integrable function $f$,
\[ \int H_V(f)\,dM = \int H_V(H_V(f))\,d\mu = \int H_V(f)\,d\mu = \int f\,dM.   \]
This shows (see Proposition \ref{fixdual}) that $M$ is Haar-invariant with Jacobian $e^{V}$. Furthermore,
for any $g$ constant on classes we have
\[ \int g(y)\,dM(y) \stackrel{M= H_V^{*}(\mu)}{=} \iint e^{V(x)} g(x)\,\hat{\nu}^y(dx)d\mu(y)\]
\[\stackrel{g(x)=g(y)}{=} \int g(y)\int e^{V(x)}\,\hat{\nu}^y(dx)\,\,d\mu(y)=\int g(y)\,d\mu(y). \]
Suppose now that $M_1$ and $M_2$ are Haar-invariant measures with Jacobian $e^{V}$ satisfying
	\[ \int g\,dM_1 = \int g\, d\mu=\int g\,dM_2,\]
for any bounded function $g$ constant on classes.
As, for any bounded function $f$, the function $H_V(f)$ is constant on classes, we get
\[\int f\, dM_1 = \int H_V(f) \,dM_1 =\int H_V(f) \,dM_2 = \int f\, dM_2.\]

\end{proof}

\begin{corollary}\label{sup} Let $V$ be a Haar-normalized function and $g$ be a measurable and bounded function which is constant on classes. Then, we have
\begin{align*}
\sup_{y\in \Omega}g(y)&= \sup\left\{ \int g\,dM\,|\, M \,is\, probability\, Haar-invariant\right\}\\
&= \sup\left\{ \int g\,dM\, | \, M \,is\, prob.\,\, Haar-invariant\, with \, Jacobian \, e^{V} \right\} .
\end{align*}

\end{corollary}
\begin{proof}
	Clearly,
	\[\sup_{y\in \Omega}g(y) \geq \sup\left\{ \int g\,dM\,|\, M \,is\, probability\, Haar-invariant\right\}. \]
	\[\geq \sup\left\{ \int g\,dM\, | \, M \, is\,probability\, Haar-invariant\, with \, Jacobian \, e^{V} \right\} .\]
	
On the other hand, for any given $\epsilon>0$, let $y_\epsilon \in \Omega$ be such that $g(y_\epsilon) + \epsilon >  \sup_{y\in \Omega}g(y)$. Let $\mu_\epsilon = \delta_{y_\epsilon}$. From above theorem there exists a probability $M_\epsilon$ Haar-invariant, with Jacobian $e^{V}$,  such that,
\[ \int g\, dM_\epsilon = \int g\, d\delta_{y_\epsilon} = g(y_\epsilon) > (\sup_{y\in \Omega}g(y)) -\epsilon.  \]
Taking the supremum over \, Haar-invariant\,probabilities with \, Jacobian \, $e^{V}$ and observing that $\epsilon$ is arbitrary we completed the proof.
	
\end{proof}

\subsection{Transverse measures and Haar-invariant probabilities} \label{trans}

General references on transverse measures are  \cite{K} and \cite{Con}. Remember that we consider fixed a certain Haar-system $(G,\hat{\nu})$ where the transverse function $\hat{\nu}$ satisfies 
$\int 1\,\,\hat{\nu}^y (dx) =1, \,\,\forall y$.

In this section we study in our setting   the relation between a transverse measure $\Lambda$ and a Haar-invariant probability $M$.  We propose a more direct and simple discussion, under the present setting, similar to the one that appears in section 5 of \cite{CLM}  but with some new proofs. The goal here is to prove Theorem \ref{equivalent} (which has a similar claim in section 5 in  \cite{CLM} but it will improved here).   We remark that in \cite{CLM} it is not considered  Haar invariant probabilities. We will start by remembering the following result stated  in \cite{CLM}.

\begin{proposition} Given a transverse function $\hat{\nu}$, a modular function $\delta(x,y)$ and a quasi-invariant probability $M$, if $\hat{\nu} * \lambda_1 = \hat{\nu} * \lambda_2 $, where $\lambda_1,\lambda_2$ are kernels, then,
$$ \int \delta^{-1} \lambda_1(1) \, \,d\, M=\int \delta^{-1} \lambda_2(1) \, \,d\, M.$$
This means
	\[\iint \delta(y,x) \lambda_1^y (dx) M(y)  =
	\iint \delta(y,x) \lambda_2^y (dx) M (y).\]
\end{proposition}

{\bf Proof:} See proposition 65 in the section 5 of \cite{CLM}.
\qed

\bigskip

As $\hat{\nu}^{y}$ is a probability, for any transverse function $\nu$ we get
\[ \int f(s,y) \,{\nu}^{y}(ds)  = \int f(s,y) \nu^{y}(ds)\hat{\nu}^{y}(dx) \]\[=\int f(s,y) \nu^{x}(ds)\hat{\nu}^{y}(dx) \stackrel{(\ref{convolution})}{=} \int f(s,y) (\hat{\nu}* \nu)^{y}(ds) .\]
This shows that $\nu= \hat{\nu}*\nu$  for any transverse function $\nu$.
Consequently, if $\lambda$ is any kernel such that $\nu= \hat{\nu}*\lambda$, then $\hat{\nu}*\nu=\hat{\nu}*\lambda$. Applying the above proposition, we get, for any quasi-invariant probability $M$ for $\hat{\nu}$ and $\delta$,
\[\nu= \hat{\nu}*\lambda\,\,\,\text{implies}\,\,\,\iint \delta(y,x) \lambda^y (dx) d M(y)  =
\iint \delta(y,x) \nu^y (dx) d M(y).\]

Given the transverse function $\hat{\nu}$, a modular function $\delta(x,y)$ and an associated quasi-invariant measure $M$, we define (see Theorem 66 in the section 5 of \cite{CLM} or  \cite{Con}) a transverse measure $\Lambda $ as
$$\Lambda (\nu) = \,\iint \delta(y,x) \lambda^y (dx)
d M(y),$$ where $\lambda$ is any kernel satisfying $\nu= \hat{\nu}*\lambda$.

From now on we will consider a modular function $\delta(x,y)=e^{V(y)-V(x)}$ where $V$ is $\hat{\nu}$-normalized and a Haar-invariant probability $M$ with Jacobian $e^{V}$. Furthermore, as in the present setting $\hat{\nu}^{y}$ is a probability,  we can take $\lambda = \nu$, and define
\begin{equation}\label{eq1}
\Lambda (\nu) = \iint e^{V(x)-V(y)} \nu^y (dx)\,d {M}(y).
\end{equation}

The first result below provides an alternative expression for $\Lambda$ in (\ref{eq1}).

\begin{proposition} Suppose that $M$ is a Haar-invariant probability associated to the Jacobian $e^V$. Then, for any transverse function $\nu$, we get
\begin{equation}\label{eq2}
 \Lambda(\nu) := \iint e^{V(x)-V(y)} \nu^y (dx)\,d {M}(y) =\iint e^{V(x)} \nu^y (dx)d {M}(y).
\end{equation}

\end{proposition}

\begin{proof}
	As $M$ is Haar-invariant (for the fixed transverse function $\hat{\nu}$) and associated to the Jacobian $e^V$, if we call $F(y)=e^{-V(y)}\int e^{V(x)}\nu^{y}(dx)$, then we get from (\ref{eq1}) and Proposition \ref{fixdual} that, for any transverse function $\nu$,
	\[ \Lambda (\nu) = \iint e^{V(x)-V(y)} \nu^y (dx)\,d {M}(y)= \int F(y)\,dM(y) \]
	\[\stackrel{H_V^*(M)=M}{=} \iint e^{V(x)}F(x)\,\hat{\nu}^{y}(dx)dM(y) \]
 \[=\iint e^{V(x)}\left[e^{-V(x)}\int e^{V(s)}\nu^{x}(ds)\right]\,\hat{\nu}^{y}(dx)dM(y) \]
	\[= \iiint e^{V(s)}\nu^{x}(ds)\,\hat{\nu}^{y}(dx)dM(y) = \iiint e^{V(s)}\nu^{y}(ds)\,\hat{\nu}^{y}(dx)dM(y) \]
	\[= \iint e^{V(s)}\nu^{y}(ds)\,dM(y) = \iint e^{V(x)}\nu^{y}(dx)\,dM(y). \]
\end{proof}

 \begin{definition}  \label{ret} We say that a transverse measure $\Lambda$ is a {\bf Haar-invariant transverse probability} if it has modulus $\delta(x,y)=e^{V(y)-V(x)}$, where $V$ is a Haar-normalized function, and, furthermore, $\Lambda(\hat{\nu}) =1$.

 	We denote by $\mathcal{M} (\hat{\nu})$ { the set of all Haar-invariant transverse  probabilities} $\Lambda$ for the Haar system $(G,\hat{\nu})$.
 	
 \end{definition}

The next result corresponds to  the Theorem 66 in  \cite{CLM}.

\begin{proposition}\label{istransm}
	Suppose that $M$ is a Haar-invariant probability associated to the Jacobian $e^V$. Then, $\Lambda$ as defined by expression (\ref{eq2}) is a  Haar-invariant transverse  probability.
\end{proposition}

\begin{proof}
Clearly $\Lambda(\hat{\nu}) =1$. We want to show that $\Lambda$ satisfies Definition \ref{coco} with $\delta(x,y)=e^{V(y)-V(x)}$.  From (\ref{eq2}) 	$\Lambda$ is linear over transverse functions. Furthermore, if $\nu_1*(\delta\lambda) =\nu_2$, with $\lambda^{y}(1)=1\,,\forall\,y$, and $\delta(x,y)=e^{V(y)-V(x)}$, then, we have
	$$ \Lambda(\nu_2) =  \iint e^{V(x)}\, \nu_2^y (dx)d {M}(y) = 	\iiint  e^{V(s)} \delta(s,x)\,\lambda^{x}(ds) \nu_1^y (dx)d {M}(y)$$
	$$= 	\iiint e^{V(s)} e^{V(x)-V(s)}\,\lambda^{x}(ds) \nu_1^y (dx)d {M}(y)= 	\iiint  e^{V(x)} \,\lambda^{x}(ds) \nu_1^y (dx)d {M}(y)$$
	$$= 	\iiint 1\, \lambda^{x}(ds) e^{V(x)} \nu_1^y (dx)d {M}(y)= 	\iint  e^{V(x)} \,\nu_1^y (dx)d {M}(y) =\Lambda(\nu_1).$$ 	 	
\end{proof}

The next two results complete the study of the relation between a Haar-invariant transverse probability $\Lambda$ and a  Haar-invariant probability $M$.

\begin{proposition}\label{density} 	Suppose that $M$ is Haar-invariant and associated to the Jacobian $e^V$. Let $\Lambda$ be the  transverse measure as defined by expression (\ref{eq2}). Consider the transverse function  $\nu^y(dx) = F(x)\hat{\nu}^y(dx)$,  where $F$ is measurable and bounded. Then,
	$$\Lambda(\nu)= \int F(x)\,\, d M(x).$$
\end{proposition}

{\bf Proof:}
As $M$ is quasi-invariant (see remark \ref{rur}) we have
$$\Lambda (\nu) := \iint e^{V(x)-V(y)} \nu^y (dx)
d {M}(y) =  \iint e^{V(x)-V(y)} F(x)\, \hat{\nu}^y (dx)
d {M}(y)$$
$$=\iint  F(y)\, \hat{\nu}^y (dx)
d {M}(y)=\int   F(y) \int 1\,\hat{\nu}^y (dx)
d {M}(y)=\int   F(y) \,
d {M}(y).$$ \qed

\begin{remark} \label{re1}
The above theorem says   that the transverse function $\nu$ is a more general concept than a function $F$ and the transverse measure $\Lambda$ is a more general concept than a measure $M$.	Note that if any class of the equivalence relation is finite, then,  given any  transverse function $\nu$, there exists a function $F$, such that,  $\nu^y(dx) = F(x)\hat{\nu}^y(dx)$.
\end{remark}

\begin{remark}
If we consider a more general density $\nu^y(dx) = F(x,y)\hat{\nu}^y(dx)$, then, $\nu^y(dx)$ is a kernel but not a transverse function, except if $F(x,y)=F(x,z)$, for any $z\in[y]$. But in this case, as $x\in [y]$, we get that $F(x,y)=F(x,x)$, that is, $F$ depends only of $x$ as in the above theorem.
\end{remark}

The next result shows us that any Haar-invariant transverse probability of modulus $\delta(x,y)=e^{V(y)-V(x)}$ is of the form (\ref{eq2}).

\begin{proposition}\label{return} Let $\Lambda$ be a Haar-invariant transverse probability for the modular function $\delta(x,y)=e^{V(y)-V(x)}$, where $V$ is Haar-normalized. Let us define a probability $M$ on $\Omega$ which satisfies, for any measurable and bounded function $F:\Omega \to \mathbb{R}$,
	\[ \int F(y)dM(y):=\Lambda(F(x)\hat{\nu}^y(dx))  .\]
Then, $M$ is a Haar-invariant probability with Jacobian $e^V$. Furthermore, for any transverse function $\nu$ we have
\[\Lambda (\nu) = \iint e^{V(x)}\nu^{y}(dx)dM(y). \]
\end{proposition}

\begin{proof}
Let $\lambda^{y}(dx) := e^{V(x)}\hat{\nu}^{y}(dx)$. Then, as $V$ is normalized, $\lambda^{y}(1) = 1 \,,\,\forall y$.	
Claim: $(\nu * (\delta \lambda))^{y}(dz) = C(z) \hat{\nu}^{y}(dz)$, where
$$C(z) := \int e^{V(x)}\nu^{z}(dx)=\int e^{V(x)}\nu^{y}(dx)=C(y)$$ is a constant function on the class of $y$.	

\bigskip
In order to prove the claim we consider any test function $f(x,y)$. Then, for each fixed $y$
$$\int f(x,y) [\nu * (\delta \lambda)]^y(dx) = \iint f(s,y) \delta(s,x)\lambda^{x}(ds)\nu^{y}(dx) $$
$$= \iint f(s,y) e^{V(x)-V(s)} e^{V(s)}\hat{\nu}^{x}(ds)\nu^{y}(dx) = \iint f(s,y) e^{V(x)}\hat{\nu}^{x}(ds)\nu^{y}(dx) $$
$$= \int [\int f(s,y) \hat{\nu}^{x}(ds)] e^{V(x)}\nu^{y}(dx) = \int [\int f(s,y) \hat{\nu}^{y}(ds)] e^{V(x)}\nu^{y}(dx) $$	
$$= [\int f(s,y) \hat{\nu}^{y}(ds)][\int  e^{V(x)}\nu^{y}(dx)] = [\int f(s,y) \hat{\nu}^{y}(ds)][C(y)] $$
$$= \int f(s,y) C(y)\hat{\nu}^{y}(ds) = \int f(x,y) C(y)\hat{\nu}^{y}(dx). $$
This proves the claim.

\bigskip

Let $M$ be defined by
\[ \int F(y)dM(y):=\Lambda(F(x)\hat{\nu}^y(dx)), \]
for any measurable and bounded function $F$. As $\Lambda$ is linear and $\Lambda(\hat{\nu})=1$ we obtain that $M$ is a probability on $\Omega$.

 As $\Lambda$ is a transverse measure  it follows from the claim (see Definition \ref{coco}) that
\[ \Lambda(\nu) = \Lambda ( C(z) \hat{\nu}^{y}(dz))  = \int C(y)\,dM(y) = \iint e^{V(x)}\nu^{y}(dx)dM(y). \]

It remains to prove that $M$ is Haar-invariant with Jacobian $e^{V}$. Let $f:\Omega \to [0,+\infty)$  be a measurable and bounded function and define $\nu^{y}(dx) =f(x)\hat{\nu}^y(dx)$. It follows from the above claim that
\[(\nu * (\delta \lambda))^{y}(dz) =  \int e^{V(x)}f(x)\hat{\nu}^{z}(dx) \hat{\nu}^{y}(dz) ,\]
and, then, as $\Lambda$ is a transverse measure we get
\[\Lambda(f(x)\hat{\nu}^y(dx)) =\Lambda(\nu)=\Lambda(\nu * (\delta \lambda))=\Lambda\left(\int e^{V(x)}f(x)\hat{\nu}^{z}(dx) \hat{\nu}^{y}(dz)\right). \]
Therefore, by definition of $M$, we finally get
\[ \int f(x) \,dM(x) = \int \int e^{V(x)}f(x)\hat{\nu}^{z}(dx) \,dM(z),  \]
which, by linearity can extend the claim for any measurable and bounded function $f$. This shows that $M$ is Haar-invariant with Jacobian $e^{V}$.
\end{proof}	

Now, we summarize the results of  this section.

\begin{theorem}\label{equivalent} Let $V$ be a Haar-normalized function. There is an invertible map that associate for each Haar-invariant probability $M$ over $\Omega$, with Jacobian $e^{V}$, a Haar-invariant transverse probability $\Lambda$ of modulus $\delta(x,y)=e^{V(y)-V(x)}$. For any given $M$ the associated $\Lambda$ obtained by this map satisfies
\[\Lambda(\nu):=\iint e^{V(x)} \nu^y (dx)d {M}(y),\,\, \nu \in \mathcal{E}^+.  \]
On the other hand, given $\Lambda$, the associated $M$ by the inverse map satisfies
\[ \int F(y)\,dM(y):=\Lambda(F(x)\hat{\nu}^{y}(dx)),\, \text{for any}\, F \, \text{measurable and bounded}. \]
\end{theorem}

\subsection{Entropy of transverse measures} \label{ent}

On this section remains fixed a Haar-system $(G,\hat{\nu})$ where the transverse function $\hat{\nu}$ satisfies $\int 1\,\hat{\nu}^y (dx) =1, \,\,\forall y$. We use the notations and  hypotheses of Theorem $\ref{equivalent}$. Remember that we denote by $\mathcal{M} (\hat{\nu})$  the set of  Haar-invariant transverse  probabilities $\Lambda$ for the Haar system $(G,\hat{\nu})$. In some sense (see Theorem \ref{equivalent}) the set $\mathcal{M} (\hat{\nu})$ corresponds in Thermodynamic Formalism (Ergodic Theory) to the set of invariant probabilities.

We will be able to extend some concepts in Ergodic Theory concerning entropy to the Haar system formalism. The transverse function $\hat{\nu}$ will play the role of the a priori probability in \cite{LMMS} (where one can found the motivation for the definition below).

Our concept  of invariant probability does not necessarily coincide with the one in classical measurable dynamics.

\begin{definition} We define the entropy of a Haar-invariant transverse probability $\Lambda$ relative to $\hat{\nu}$ (or relative to $(G,\hat{\nu})$)  as
\[ h_{\hat{\nu}}(\Lambda) = -\sup \{ \Lambda(F(x)\hat{\nu}^{y}(dx))\,|\, F \,\,\text{is Haar-normalized} \}.\]
\end{definition}
If $\Lambda$ has modulus $\delta(x,y)=e^{V(y)-V(x)}$, where $V$ is Haar-normalized, and $M$ is the corresponding Haar-invariant probability given in Theorem \ref{equivalent}, then we get
\[ h_{\hat{\nu}}(\Lambda) = -\sup \{ \int F(x)\,dM(x)\,|\,  F \,\,\text{is Haar-normalized}\}.\]
As we define entropy just for transverse measures in the set $\mathcal{M} (\hat{\nu})$, it follows from Theorem \ref{equivalent} that we are defining similarly  entropy for any Haar-invariant probabilities.

\medskip

\begin{theorem}\label{entropyjacobian} Suppose $\Lambda \in \cal M(\hat{\nu})$ has modulus $\delta(x,y)=e^{V(y)-V(x)}$, where $V$ is a Haar-normalized function. Then,
\begin{equation} \label{realent} h_{\hat{\nu}}(\Lambda) = -\int V(x)\,dM(x) = -\Lambda(V\hat{\nu}), \end{equation}
where $M$ is defined in Theorem \ref{equivalent}.
\end{theorem}

\begin{proof} (This proof follows ideas from \cite{LMMS})  By construction $M$ is Haar-invariant with Jacobian $e^{V}$. The second equality on expression (\ref{realent}) is a consequence of Theorem \ref{equivalent}. In order to prove the first equality  we consider a general Haar-normalized function $U$, and then, we claim  that
\begin{equation}\label{jacobian}
 \int U(x)\,dM(x) \leq \int V(x)\,dM(x) .
\end{equation}	
From this inequality we obtain
 \[\sup \{\int F(x)\,dM(x)\,|\, F\, \text{is Haar-normalized} \}	=  \int V(x)\,dM(x),\]
which proves that
\[h_{\hat{\nu}}(\Lambda) = -\int V(x)\,dM(x).\]
In order to prove (\ref{jacobian}) we consider again the operator
\[H_V(f)(y) = \int e^{V(x)}f(x)\,d\hat{\nu}^{y}(dx).  \]
The probability $M$ satisfies $H_V^*(M)=M$  according to Proposition \ref{fixdual}.

Let $u=e^{U-V}$. Then, $u\,e^{V-U}= 1$, and moreover $H_V(u)(y)=1 $, for any $y$. It follows that
\[ 0 = \log(1/1) = \log\left( \frac{H_V(u)}{ue^{V-U}}\right) = \log({H_V(u)}) - \log(u) +U-V,  \]
and
\[ 0 = \int \log({H_V(u)})\,dM - \int \log(u)\,dM +\int U\,dM -\int V \,dM.   \]
Therefore,
\[ \int V\,dM - \int U\,dM = \int \log({H_V(u)})\,dM - \int \log(u)\,dM \]
\[   = \int \log({H_V(u)})\,dM - \int H_V(\log(u))\,dM \geq 0, \]
because for any $y$ we can consider the probability $e^{V(x)}\hat{\nu}^{y}(dx)$ and apply the Jensen's inequality, in the following way,
\small
\[\log({H_V(u)}) = \log \left(\int u(x)e^{V(x)}\hat{\nu}^{y}(dx)\right)  \geq \int \log(u)(x)e^{V(x)}\hat{\nu}^{y}(dx) =  H_V(\log(u)) .\] \normalsize
\end{proof}

\begin{proposition} The entropy above defined has the following properties:\newline
1. $h_{\hat{\nu}}(\Lambda)\leq 0$	for any $\Lambda \in \mathcal{M}(\hat{\nu})$\newline
2. $h_{\hat{\nu}}(\cdot) $ is concave\newline
3. $h_{\hat{\nu}}(\cdot)$ is upper semi continuous. More precisely, if $\Lambda_n(\nu)\to \Lambda(\nu)$, for any transverse function $\nu$, then, $\displaystyle{\limsup_n h_{\hat{\nu}}(\Lambda_n) \leq h_{\hat{\nu}}(\Lambda).}$
\end{proposition}

\begin{proof} -\newline
Proof of 1.: Just take $V=0$ which is Haar-normalized.\newline
Proof of 2.: Suppose $\Lambda = a_1\Lambda_1 + a_2 \Lambda_2$, where $\Lambda_1$ and $\Lambda_2$ are Haar-invariant transverse probabilities, $a_1,a_2\geq 0$ and $a_1+a_2= 1$. Then,
 \[ 	 h_{\hat{\nu}}(\Lambda) = -\sup \{\Lambda(F\hat{\nu})\,|\, F \,\text{normalized} \, \}\]
 \[       = \inf  \{\Lambda(-F\hat{\nu})\,|\, F \,\text{normalized} \, \}\]
 \[      = \inf  \{a_1\Lambda_1(-F\hat{\nu})+a_2\Lambda_2(-F\hat{\nu})\,|\, F \,\text{normalized} \, \}\]
\[      \geq  a_1\inf  \{\Lambda_1(-F\hat{\nu})\,|\, F \,\text{normalized} \, \}+a_2\inf\{\Lambda_2(-F\hat{\nu})\,|\, F \,\text{normalized} \, \}\]
\[= a_1h_{\hat{\nu}}(\Lambda_1)+a_2h_{\hat{\nu}}(\Lambda_1). \]	
Proof of 3.: Let $V$ be a Haar-normalized function, such that, $\Lambda$ has modulus $e^{V(y)-V(x)}$. Given any $\epsilon>0$, we have that
$\Lambda_n(-V\hat{\nu}) \leq \Lambda(-V\hat{\nu}) +\epsilon =h_{\hat{\nu}}(\Lambda)+\epsilon$, for sufficiently large $n$. Then, for sufficiently large $n$ we get
\[h_{\hat{\nu}}(\Lambda_n) = \inf \{\Lambda_n(-F\hat{\nu})\,|\, F \,\text{normalized} \, \} \leq \Lambda_n(-V\hat{\nu}) \leq  h_{\hat{\nu}}(\Lambda)+\epsilon. \]

\end{proof}

\bigskip

\subsection{Pressure of transverse functions}

In this section remains fixed a Haar system $(G,\hat{\nu})$ where the transverse function $\hat{\nu}$ satisfies $\int \hat{\nu}^y (dr) =1, \,\,\forall y$.
\medskip

\begin{definition} \label{prepre} We define the {\bf $\hat{\nu}$-Pressure of the transverse function $\nu$ } by
\begin{equation} \label{preo} P_{\hat{\nu}}(\nu) = \sup_{\Lambda\, \,\in \mathcal{M} (\hat{\nu})} \,\{\,\Lambda(\nu) +  h_{\hat{\nu}}(\Lambda)\,\}.\end{equation}
A transverse measure $\Lambda \in \mathcal{M}(\hat{\nu})$ which attains the supremum on the above expression will be called an {\bf equilibrium transverse measure} for the {\bf transverse function $\nu$}.
\end{definition}

\bigskip

 In the following we use the notation $M_V$ to denote a Haar-invariant probability with Jacobian $e^{V}$. Remember from Theorem \ref{Mandmu} that there are several of such associated probabilities.

\bigskip

 It follows from theorems \ref{equivalent} and \ref{entropyjacobian} that
 \[ P_{\hat{\nu}}(\nu) = \sup_{V \,\hat{\nu}-normalized}\,\,\,\,\,\sup_{M_V} \left[\int e^{V(x)}\nu^{y}(dx)dM_V(y)-  \int V(y)dM_V(y) \right] , \]
which, from proposition \ref{fixdual}, can be rewritten as \footnotesize
\[ P_{\hat{\nu}}(\nu) = \sup_{V \,\hat{\nu}-normalized}\,\,\,\,\,\sup_{M_V}\left[ \int e^{V(x)}\nu^{y}(dx)dM_V(y)-  \int e^{V(x)}V(x)\hat{\nu}^y(dx)dM_V(y)\right].  \] \normalsize

\medskip

The cases where $\nu$ is of the form $\nu = U(x)\,\hat{\nu}^{y}(dx)$ are studied below. In these particular cases is natural to interpret $\nu$ as the function $U:\Omega \to \mathbb{R}$ and, using Theorem \ref{equivalent}, to interpret $\Lambda$ as the associated probability $M$ on $\Omega$.

\begin{proposition}\label{pressurezero} Suppose that $U$ is $\hat{\nu}$-normalized. Consider the transverse function $\nu = U(x)\,\hat{\nu}^{y}(dx)$. Then, $P_{\hat{\nu}}(\nu) = 0$. If $M_U$ is any Haar-invariant probability with Jacobian $e^{U}$ and $\Lambda_U$ is the associated transverse measure from Theorem \ref{equivalent}, then $\Lambda_U$ is an equilibrium for $\nu$.
\end{proposition}

\begin{proof}\footnotesize
	\[ P_{\hat{\nu}}(\nu) = \sup_{V \,\hat{\nu}-normalized}\sup_{M_V}\left[ \int e^{V(x)}U(x)\hat{\nu}^{y}(dx)dM_V(y)-  \int e^{V(x)}V(x)\hat{\nu}^y(dx)dM_V(y)\right]  \]
	\[\stackrel{H_V^*(M_V)=M_V}{=} \sup_{V \,\hat{\nu}-normalized}\sup_{M_V}\left[ \int U(y)dM_V(y)-  \int V(y)dM_V(y)\right].  \] \normalsize
From (\ref{jacobian}) the last expression is smaller then zero and, on the other hand, taking $V=U$, as a particular $V$ under the supremum,  the last expression is equal to zero. Then, the choices $V=U$ and any probability $M_U$ attain the supremum.
\end{proof}

For the next result we suggest  the reader to recall (in advance) the claim of Proposition \ref{eigenvalue} before reading it.

\medskip

\begin{proposition} Consider the transverse function $\nu = U(x)\hat{\nu}^{y}(dx)$, where $U$ is measurable and bounded, but not necessarily $\hat{\nu}$-normalized. Suppose that $\tilde U(y) = \int e^{U(x)}\, \hat{\nu}^y(dx)$ is a constant function, $\tilde U(y) = \lambda$ for all $y$.
Then, \[P_{\hat{\nu}}(\nu) = \log(\lambda).  \]
\end{proposition}

\begin{proof}
	From Proposition \ref{eigenvalue} the function $U-\log(\lambda)$ is normalized. Then,
\footnotesize	\[P_{\hat{\nu}}(\nu) = \sup_{V \,\hat{\nu}-normalized}\sup_{M_V}\left[ \int e^{V(x)}U(x)\hat{\nu}^{y}(dx)dM_V(y)-  \int e^{V(x)}V(x)\hat{\nu}^y(dx)dM_V(y)\right]  \] \normalsize
	\[=  \sup_{V \,\hat{\nu}-normalized}\sup_{M_V}\left[\int U(y)\,dM_V(y)-  \int V(y) \,dM_V(y)\right]  \]
	\[= \sup_{V \,\hat{\nu}-normalized}\sup_{M_V}\left[ \int [U-\log(\lambda)dM_V-  \int VdM_V + \log(\lambda) \right]   = \log(\lambda), \]
	where the last equality is a consequence of (\ref{jacobian}) together with the fact that we can take also $V=U-\log(\lambda)$ under the supremum.
	\end{proof}

Now we consider the general case where $\tilde U$ is not constant.
\medskip

\begin{proposition}\label{pressuretilde} Consider the transverse function $\nu = U(x)\hat{\nu}^{y}(dx)$ where $U$ is bounded and measurable, but not necessarily normalized. Let $\tilde U(y) = \int e^{U(x)}\, \hat{\nu}^y(dx)$.
Then, \[P_{\hat{\nu}}(\nu) = \sup_{y\in \Omega}\, [\log (\tilde U (y))] . \]

\end{proposition}
\begin{proof}
First note that
	\[P_{\hat{\nu}}(\nu) =  \sup_{V \,\hat{\nu}-normalized}\sup_{M_V}\left[\int U(y)\,dM_V(y)-  \int V(y) \,dM_V(y)\right]  \]
	\[= \sup_{V \,\hat{\nu}-normalized}\sup_{M_V}\left[ \int [U-\log(\tilde U)dM_V-  \int VdM_V + \int \log(\tilde U)dM_V \right] . \]
On the one hand, from (\ref{jacobian}), as  $U-\log(\tilde U)$ is normalized we get $$P_{\hat{\nu}}(\nu) \leq \sup_{V \,\hat{\nu}-normalized}\sup_{M_V}\left[ \int \log(\tilde U)dM_V \right] = \sup_{M \, Haar-invariant} \int \log(\tilde U)dM .$$
On the other hand, choosing $V=U-\log(\tilde U)$ we get
$$P_{\hat{\nu}}(\nu) \geq \sup_{M_{U-\log \tilde U}}  \int \log(\tilde U)dM_{U-\log \tilde U}.$$

As $\tilde U$ is constant on classes  then from Corollary \ref{sup} we conclude the proof.
\end{proof}

\begin{remark} If there exists $y_0 \in \Omega$ satisfying  \,\,$\sup_{y\in \Omega}\, [\log (\tilde U (y))] = \log(\tilde U(y_0))$, then, taking $\mu=\delta_{y_0}$ and applying Theorem \ref{Mandmu}, there exists a Haar-invariant probability $M_{U-\log(\tilde U)}$ satisfying
\[ P_{\hat{\nu}}(\nu)= \int \log(\tilde U)\,d\delta_{y_0} = \int \log(\tilde U)\,dM_{U-\log(\tilde U)}. \]
In this case, if $\Lambda$ is the transverse measure associated to $M_{U-\log(\tilde U)}$ by Theorem \ref{equivalent}, then $\Lambda$ is an equilibrium for $\nu$.
\end{remark}

\bigskip

 We observe that $P_{\hat{\nu}}(\cdot)$ plays the role of the Legendre's transform of $-h_{\hat{\nu}}(\cdot)$. As $-h$ is convex it is natural to expect an involution, that is
\[ -h_{\hat{\nu}}(\Lambda) = \sup_{\nu} [\Lambda(\nu) - P_{\hat{\nu}}(\nu)] ,\]
or, equivalently,
\[h_{\hat{\nu}}(\Lambda) = \inf_{\nu} [-\Lambda(\nu) + P_{\hat{\nu}}(\nu)]. \]

\begin{proposition} The $\hat{\nu}$-Pressure of the transverse function $\nu$  and the entropy
 of the Haar-invariant transverse probability $\Lambda$     are related by the expression:

\[h_{\hat{\nu}}(\Lambda) = \inf_{\nu} [-\Lambda(\nu) + P_{\hat{\nu}}(\nu)]. \]
\end{proposition}

\begin{proof}
We observe that for any $\nu$ we have, by definition of pressure,
\[h_{\hat{\nu}}(\Lambda) \leq  [-\Lambda(\nu) + P_{\hat{\nu}}(\nu)]. \]
It follows that
\[h_{\hat{\nu}}(\Lambda) \leq \inf_{\nu} [-\Lambda(\nu) + P_{\hat{\nu}}(\nu)]. \]
On the other hand, if  $\Lambda$ as modulus $e^{V(y)-V(x)}$, where $V$ is $\hat{\nu}$-normalized, then, taking $\mu^y(dx) = V(x)\hat{\nu}^{y}(dx)$ we get, from Theorem \ref{entropyjacobian} and Proposition \ref{pressurezero},
\[ h_{\hat{\nu}}(\Lambda) = -\Lambda(\mu)\,\,\, and\,\,\, P_{\hat{\nu}}(\mu) =0.\]
Therefore,
\[h_{\hat{\nu}}(\Lambda) = -\Lambda(\mu) + P_{\hat{\nu}}(\mu)  \geq \inf_{\nu} [-\Lambda(\nu) + P_{\hat{\nu}}(\nu)].\]
\end{proof}

\section{Examples}\label{examples}

\subsection{Entropy and pressure in the XY model}

We consider the hypotheses and notations of section \ref{insp}.  In this case it is easy to see that
\[ H_V(f)(y) = \mathcal{L}_V(f)(\sigma(y)),\]
where $ \mathcal{L}_V$ is the classical Ruelle operator. We say that a bounded and measurable potential $V$ is normalized if it satisfies  $\mathcal{L}_V(1)=1$.

\begin{proposition} $V$ is normalized, iff, $V$ is Haar-normalized.
\end{proposition}
\begin{proof}
If $V$ is normalized then for any $y\in \Omega$ we have
\[ \int e^{V(x)}\,\hat{\nu}^y(dx) = H_V(1)(y) = \mathcal{L}_V(1)(\sigma(y)) =1 \]
which proves that $V$ is Haar-normalized.

If $V$ is Haar-normalized, then for any given $y\in \Omega$ we choose $z\in \Omega$, such that, $\sigma(z)=y$. It follows that
\[L_V(1)(y) =L_V(1)(\sigma(z)) = H_V(1)(z) =1\,,\, \forall\,\, y\in \Omega. \]
\end{proof} 	

\begin{proposition} If a probability measure $M$ on $\Omega$ satisfies $\mathcal{L}_V^*(M) =M$, for some normalized H\"{o}lder function $V$, then, it is Haar-invariant.
\end{proposition}

\begin{proof} From propositions \ref{mimeo}  and \ref{fixdual} we conclude that $M$ is Haar-invariant with Jacobian $e^{V}$.
\end{proof}

The identification of $\sigma$-invariant probabilities and Haar-invariant probabilities is in general false. There exists Haar-invariant probabilities which are not invariant for the shift map and reciprocally, as  next example shows.

\begin{example}\label{diferent}  Let $\mu$ be any probability measure on $\Omega$, such that the push-forward probability $\nu$ defined from $\int f\,d\nu := \int f\circ\sigma \,d\mu$ is not invariant for the shift map\footnote{For instance, $\mu=\delta_{(1,1,0,0,0,0...)}$}. From Theorem \ref{Mandmu} there exists a Haar-invariant probability $M$, such that, $\int f\circ \sigma \,dM = \int f\circ \sigma\,d\mu$, for any measurable and bounded function $f$ (because $f\circ\sigma$ is constant on classes). If $M$ were invariant for the shift map, then  $$\int f\,dM=\int f\circ\sigma\,dM = \int f\circ \sigma\,d\mu= \int f\,d\nu,$$ which is a contradiction, because $\nu$ is not invariant.

\medskip

On the other hand, there are shift-invariant probabilities which are not Haar-invariant for a fixed Haar-system $(G,\hat{\nu})$. For instance, consider $M=\delta_{0^{\infty}}$ where $0^{\infty}=(0,0,0,...) \in K^{\mathbb{N}}=[0,1]^{\mathbb{N}}$. In this case, supposing by contradiction that $M$ is Haar-invariant, where, for each $x$, $\hat{\nu}^{x}$ is identified with the Lebesgue measure $m$ on $[0,1]$, there must be a measurable and bounded function $V$ such that for any measurable and bounded function $f$
\[ \int f(a0^{\infty})e^{V(a0^{\infty})} dm(a) = f(0^{\infty}).\]
This is impossible because, as functions of $f$, the value on the right side can be easily changed without affecting the mean of the left side.

\end{example}

\begin{proposition} Let $M$ be the equilibrium measure for a normalized function $V$. Let $\Lambda$ be the transverse measure defined by (\ref{eq2}). Then, $h_{\hat{\nu}}(\Lambda) = -\int V\, dM.$
\end{proposition}

\begin{proof} The claim easily follows from Theorem \ref{entropyjacobian}.
\end{proof}		

The above proposition shows that the Haar-entropy is a natural generalization for Haar-Systems of the Kolmogorov-Sinai entropy.
We refer to \cite{LMMS} for a discussion regarding Kolmogorov-Sinai entropy and the concept of negative entropy for the generalized $XY$ model.

\bigskip

 The concept of pressure is very different when considered the Thermodynamic Formalism setting instead of Haar Systems. As an example of this fact, note that in Thermodynamic Formalism we have
$$ P_{\sigma}(f\circ \sigma)=P_{\sigma}(f)= \sup_{\mu \, invariant} [\int f\,d\mu + h_m(\mu)],$$ for any continuous function $f$. On the other hand, in Haar Systems, if we consider a transverse function $\nu$ in the form $\nu= U\hat{\nu}$, where $U=f\circ\sigma$, then $U$ is constant on classes. It follows that $\log(\tilde U) = U$, and, from Proposition \ref{pressuretilde}, we get
$$P_{\hat{\nu}}((f\circ\sigma)\hat{\nu})= P_{\hat{\nu}}(U\hat{\nu})= \sup_{y\in \Omega}\, [U(y)] =\sup_{y\in \Omega}\, [f(\sigma(y))]=\sup_{z\in \Omega}\, [f(z)].$$
 But, if we consider the transverse function $f\hat{\nu}$,
\[P_{\hat{\nu}}(f\hat{\nu}) = \sup_{y\in \Omega}\, [\log\int e^{f(x)} \,\hat{\nu}^y(dx)] = \sup_{y\in \Omega}\, [\log\int e^{f(a,y_2,y_3,...)} \,dm(a)]. \]

The main reason for the difference between the two kinds of pressure is in some sense described in Example \ref{diferent}. When we consider the Haar-entropy for a different set of probabilities and then consider the pressure, as a ``Lengendre's transform'' of $-h$ (which is defined over this different set), it is natural to get a different meanings for pressure.

\subsection{Haar-systems dynamically defined} \label{haar-dy}

 Assume that $\Omega$ is a complete and separable metric space and $\mathcal{B}$ denotes the Borel sigma algebra  on $\Omega$. In this section we generalize results of section \ref{insp}.

Suppose that $T:\Omega \to \Omega$ is a continuous map and consider the groupoid $G$ defined be the equivalence relation $x\sim y$, if and only if, $T(x)=T(y)$. In this way any class is closed and any transverse function $\hat{\nu}$ (which is a probability on each class) can be identified as a choice of a probability $m_y$ over the set $T^{-1}(y)$, for each $y\in \Omega$.  For any measurable and bounded function $U:\Omega \to\mathbb{R}$ and transverse function $\hat{\nu}$ we define the generalized Ruelle Operator
\[\mathcal{L}_U(f)(y)= \int_{T(x)=y} e^{U(x)}f(x) \nu^{x}(dx)=\int_{T(x)=y} e^{U(x)}f(x)\,dm_y(x), \]
where $\nu^{x}=\nu^{z}$, iff, $T(x)=T(z)$, that is, $\nu^x$ is a probability $m_y$, if $x\in T^{-1}(y)$.

We say that a measurable and bounded function $V$ is normalized if
\[\int_{T(x)=y} e^{V(x)} \nu^{x}(dx)=1 \,,\,\, \forall y\, \in \Omega.\]

\begin{proposition} Under above hypotheses and notations, suppose that an invariant probability $M$ for $T$ satisfies $\mathcal{L}_V^{*}(M)=M$, for some normalized (measurable and bounded) function $V$. Then, $H_V^{*}(M)=M$, that is, $M$ is Haar-invariant.
\end{proposition}
\begin{proof}
	Observe that $H_V(f)(y) = \mathcal{L}_V(f)(T y)$. Then,
	\[\int H_V(f)(y)\, dM(y)   =  \int \mathcal{L}_V(f)(T y)\, dM(y)  \]
	\[\stackrel{M\, is\, T\,invariant }{=}     \int \mathcal{L}_V(f)(y)\, dM(y) = \int f(y)\,dM(y) . \]
\end{proof}

The definition of entropy $h_{\hat{\nu}}$ in this work can be applied to the case of any invariant measure $M$ satisfying $\mathcal{L}_V^{*}(M)=M$, for some $V$ normalized. Such $M$ is associated with a transverse probability $\Lambda$ by Theorem \ref{equivalent}.

\bigskip

In section 2 the a priori probability $m_y$ is a fixed probability $m$ independently of $y$ in a natural way, because in that example the pre-images of any point $y$ are identified with a fixed set $K$, where $\Omega=K^{\mathbb{N}}$. Observe that for a general dynamic system there is not a natural identification of pre images of different points, that is, the sets $T^{-1}(y)$ and $T^{-1}(z)$ can be of quite distinct nature. One of the simplest examples of such kind are subshifts of finite type, where distinct points can have sets of pre-images with different cardinalities. In the present general case, in contrast with the XY model (described before), it is natural to take as an a priori probability a general transverse function.

\bigskip

In contrast with Example \ref{diferent}, the next theorem shows that any $T-$invariant probability can be seen as a Haar-invariant probability. The items 1. and 2. of the theorem say that $\mu^{x}$ is some kind of kernel, when considered almost every point ($M$-a.e.) $x\in \Omega$ (see Definition \ref{kernel}). The item 3. says that this ``kernel'' is a transverse function and the item 4. says that $M$ is ``Haar-invariant'' with Jacobian $J=e^V=1$.  In the proof we use the Rokhlin's disintegration theorem. A reference for this topic is  chapter 5 in \cite{OV}.

\begin{theorem}\label{Rokhlin} Let $\Omega$ be a complete and separable metric space and $T:\Omega \to \Omega$ be a continuous map. Consider the groupoid $G$ defined from the equivalence relation $x\sim z$, if and only if, $T(x)=T(z)$.

Then, for any fixed  $T$-invariant  probability $M$ on $\Omega$, there exists a family $\{\mu^{x}\,|\, x\in \Omega\}$ of probabilities on $\Omega$ satisfying: \newline
	1. $\mu^x$ has support on $[x]$ for $M$-a.e. $x\in \Omega$;\newline
	2. for each measurable set $E\subseteq \Omega$, the map $x\to \mu^{x}(E)$ is measurable;\newline
	3. $\mu^{x}=\mu^{z}$ for any $x,z\in\Omega$ satisfying $x\sim z$;\newline
	4. $\int f(x)\,dM(x) = \iint  f(z)\,\mu^{x}(dz)dM(x)$ for any measurable and bounded function $f:\Omega\to \mathbb{R}$.
\end{theorem}

\begin{proof}  	
	As $\Omega$ is a complete and separable metric space there exists an enumerable base of open sets $A_1,A_2,A_3,...$. This means that for each point $x\in \Omega$ and open set $U$ containing $x$ there exists some $A_i$ satisfying $x\in A_i \subseteq U$. Let $\mathcal{P}$ be the partition of $\Omega$ defined in the following way: $x$ and $z$ belong to the same element of the partition, if and only if, $\chi_{A_i}(T(x))=\chi_{A_i}(T(z))$, for any $i \in \mathbb{N}$. We observe that $\mathcal{P}$   is the partition of $\Omega$ in the classes of $G$, that is, two points $x$ and $z$ are on the same element of the partition $\mathcal{P}$, if and only if, $T(x)=T(z)$. Indeed, clearly $x\sim y$ implies that $x$ and $y$ belong to the same element of the partition. Reciprocally, if $T(x)\neq T(y)$, there exists an open set $U$ such that $T(x)\in U$ and $T(z)\notin U$. It follows that for some $A_i$ we have  $T(x)\in A_i\subseteq U$ and $T(z)\notin A_i$, which proves that $x$ and $y$ belong to different elements of the partition.
	
	We claim that $\mathcal{P}$ is a measurable partition. Indeed, it's only necessary to consider the partitions $\mathcal{P}_n$, $n\in \mathbb{N}$, defined in the following way: two points $x$ and $z$ belong to the same element of the partition $\mathcal{P}_n$, if and only if, $\chi_{A_i}(T(x))=\chi_{A_i}(T(z))$, for any $i \in \{1,...,n\}$. Observe that $\mathcal{P}_n$ has $2^{n}$ elements,
	\[\mathcal{P}_1 \prec \mathcal{P}_2\prec  \mathcal{P}_3\prec ... \]
	and $\mathcal{P} =  \bigvee_{n\geq 1} \mathcal{P}_n$. This proves the claim.

	Remember that two points $x$ and $z$ are on the same element of the partition $\mathcal{P}$, if and only if, $T(x)=T(z)$.
	We denote by $P_y$ the element of the partition $\mathcal{P}$ that contains the pre images of $y$, that is, $P_y=\{x\in \Omega\,|\,T(x)=y\}$. Observe that we can identify  $\Omega$ with $\mathcal{P}$ from $y\to P_y$. We define $\pi: \Omega \to \mathcal{P}$ by the rule: $\pi(x)$ is the element of the partition $\mathcal{P}$ that contains $x$. In this way $\pi(x)=P_y$, if and only if, $T(x)=y$. We say that $\mathcal{Q}\subseteq \mathcal{P}$ is measurable if the set $\pi^{-1}(\mathcal{Q})$ is a measurable subset of $\Omega$. For a given invariant probability $M$ on $\Omega$ we associate a probability $\hat{M}$ on $\mathcal{P}$ by
	\[\hat{M}(\mathcal{Q}) := M(\pi^{-1}(\mathcal{Q})), \]
	where $\mathcal{Q}\subseteq \mathcal{P}$ is measurable. Observe that using the identification  $y\to P_y$, for any given measurable subset $Q\subset \Omega$ we can associate the measurable subset $\mathcal{Q}=\{P_y\,|\, y\in Q\}$ of $\mathcal{P}$. Furthermore, as $M$ is $T$-invariant,
	\[\hat{M}(\mathcal{Q}) = M(\pi^{-1}(\mathcal{Q})) = M(\pi^{-1}(\{P_y\,|\, y\in Q\}) = M( T^{-1} (Q)) {=} M(Q) . \]

	\bigskip
	
	As the metric space $\Omega$ is complete and separable and the partition $\mathcal{P}$ is measurable, from Rokhlin's disintegration theorem (see \cite{OV}), any invariant probability $M$ admits a disintegration, which is a family of probabilities $\{m_P\,|\, P\in \mathcal{P}\}$ on $\Omega$ satisfying, for any measurable set $E\subset \Omega$, \newline
	1. \, $m_{P}(P) = 1$,\, for $\hat{M}-$a.e. $P\in \mathcal{P}$ \newline
	2. \, $P \to m_{P}(E)$ is measurable\newline
	3. \, $M(E) = \int m_{P}(E)\, d\hat{M}(P)$
	
	\bigskip
	
	Using the identification $y\to P_y$ we obtain a family of probabilities $\{m_y\,|\, y\in \Omega\}$ on $\Omega$ satisfying, for any measurable set $E\subset \Omega$, \newline
	1'. \, $m_{y}(T^{-1}(y)) = 1$ for $M-$ a.e. $y\in \Omega$ \newline
	2'. \, $y \to m_{y}(E)$ is measurable\newline
	3'. \, $M(E) = \int m_{y}(E)\, dM(y).$
	
	\bigskip

	We define, for each $x\in \Omega$, the probability $\mu^{x}:= m_{T(x)}$, that is,  $\mu^{x}:= m_{y}$ if $x \in T^{-1}(y)$.
	By construction, $x\sim z$ implies $\mu^{x}=\mu^{z}$. As $M$ is $T$-invariant, it follows from 1'. that $\mu^{x}([x])=1$, for
	$M-$ a.e. $x\in \Omega$. The sets $[x]=T^{-1}\{y\}$ are closed, because $T$ is continuous, therefore $\mu^x$ has support on $[x]$ for $M$-a.e. $x\in \Omega$. Furthermore, for a fixed measurable set $E$, as $x\to y=T(x)$ and $y\to m_y(E)$ are measurable maps, we obtain that $x\to \mu^{x}(E)$ is measurable.
	
	In order to conclude the proof it remains to prove the item 4. of the theorem. For any measurable and bounded function $f:\Omega \to \mathbb{R}$ we have, from 3'. and using the fact that  $M$ is $T$-invariant,
	\[ \int f(x) \, dM(x) =\iint f(z) \,dm_{x}(z) d{M}(x)= \iint f(z) \,dm_{T(x)}(z) d{M}(x)\]
	\[= \iint f(z) \,\mu^x (dz) dM(x).    \]
	
\end{proof}

\bigskip

In the next corollary we suppose that any class is finite and we consider the transverse function which is the counting measure on each class. We remark that it is a finite measure but not a probability. Anyway, an easy normalization is sufficient  to get a probability, that is, to replace  $\sum_{T(x)=T(z)}$ by $\frac{1}{\#[x]}\sum_{T(x)=T(z)}$ and $J(z)$ by $\tilde{J}(z):=(\#[x])\cdot J(z)$.

\bigskip

\begin{corollary}
	Suppose that $M$ is a complete and separable metric space and suppose that the continuous map $T:M\to M$ is such that any point $y$ has a finite number of pre images. Then, for any $T$-invariant probability $M$ there exists a bounded function $J$ defined for $M-$a.e. $x\in \Omega$ (a Haar-Jacobian of $M$) satisfying, for $M-$a.e. $x \in \Omega$,  $\sum_{T(x)=T(z)} J(z) = 1$ and
	\[ \int f(x)\,dM(x) = \int \sum_{T(x)=T(z)} J(z)f(z)\, dM(x).  \]
\end{corollary}

\begin{proof}
	Using the notations of the proof of the above theorem, as for $M$-a.e.  $x$ the probability $\mu^x$ has support in the finite set $[x]$, there exists, for any such $x$, a function $J^x$ defined over the class of $x$ satisfying $\mu^x(\{z\})=J^x(z)$, for any $z\sim x$. We define a function $J$ a.e. by $J(z)=J^{x}(z)$, if $x\sim z$, and $\mu^x([x])=1$ otherwise. The images of $J$ belongs to $[0,1]$, clearly $\sum_{T(x)=T(z)} J(z) = 1$ and, furthermore, from the above theorem and the definition of $J$ it follows that
	\[ \int f(x)\,dM(x) = \iint f(z) \,\mu^x (dz) dM(x)  =  \int \sum_{T(x)=T(z)} J(z)f(z)\, dM(x).\]   	
\end{proof}

\begin{example} If $\Omega \subseteq \{1,...,d\}^{\mathbb{N}}$ is a subshift of finite type, defined from an aperiodic matrix, and $M$ is an invariant probability for the shift map $\sigma$, then for $M$-a.e. $x\in \Omega$, $x=(x_0,x_1,x_2,...)$, there exists
	\[J(x)= \lim_{n\to\infty} \frac{M([x_0,x_1,...,x_n])}{M([x_1,x_2,...,x_n])}.\]
	In Thermodynamic Formalism this function $J$ (called Jacobian of the measure or, sometimes the inverse of the Jacobian) is $M-$integrable and for any measurable and bounded function $f:\Omega \to \mathbb{R}$ it satisfies
	\[ \int f(y)\, dM(y) = \int \sum_{\sigma(z)=y} J(z)f(z)\, dM(y). \]
	As $M$ is $\sigma$-invariant, for any measurable and bounded function $f$,
	\[ \int f(y)\, dM(y) =  \int \sum_{\sigma(z)=\sigma(y)} J(z)f(z)\, dM(y). \]
	Therefore, $M$ is also Haar-invariant with Haar-Jacobian $J$.
	
	The Kolmogorov-Sinai entropy of $M$ is given by
	\[ h_\sigma(M) = -\int \log(J)\, dM, \]
	which is compatible with the definition of Haar-entropy of $M$, introduced in this work.
\end{example}  	

Note that in the case of the groupoid of section \ref{insp} (taking $V$ continuous and positive and assuming that the equivalence classes are finite) if $\rho$ is an eigenprobability for the operator  $H_V^*$ associated to the eigenvalue $\lambda>0$, then the condition (\ref{leo}) is true for any cylinder set $B$. Indeed, $\lambda\, \int I_B d \rho = \int \mathcal{L}_V \,(I_B) (\sigma)\, d \rho $.

\subsection{Extremal cases}

In this section we suppose that $\Omega$ is measurable and consider as examples two extremal cases which  are: 1) the case  where $[x]=\Omega \,,\,\forall \,x\in \Omega$, and 2) the case where $[x]=\{x\}\,,\, \forall\,x\in\Omega$.
 We will explore on these examples the meaning of the theoretical results we get before. In this procedure we will recover some classical concepts which are well known on the literature. This shows that our reasoning is quite justifiable.

\begin{example}\label{tudo} Consider the case $[x]=\Omega$ for any $x \in \Omega$, that is
 $x\sim y$, for any $x,y \in \Omega$.  In this case, the transverse functions are the measures on $\Omega$ and we fix a probability $m$ (that plays the role of the transverse function $\hat{\nu}$ on $\Omega$).
 The Haar System $(G,m)$ will remain fixed on this example.

 A function $V$ is Haar-normalized if
\[\int e^{V(x)} dm(x) = 1.\]
For any function $f$ we have that $H_U(f)$ is constant and equal to $\int e^{U(x)} f(x) dm(x)$.
A probability $M$ is Haar-invariant with Jacobian $e^{V}$, if and only if, $dM=e^{V}dm$, because  $H_V^{*}(M)=M$ means
\[\int f(x) e^{V(x)}  dm(x) = \int f(y)dM(y)\,,\,\, \forall\, f.\]

\medskip

For a fixed Haar-invariant probability $M$ with Jacobian $e^{V}$ we associate the transverse measure $\Lambda$ acting on measures as
\[\Lambda(\nu) = \int e^{V(x)}d\nu(x)dM(y) = \int e^{V(x)}d\nu(x).\]
On this way, it is more natural to consider that for a Haar-normalized function $V$  we associate to it the above $\Lambda$ - which is the unique transverse probability for the modular function $\delta(x,y)=e^{V(y)-V(x)}$.

\medskip

The entropy of $\Lambda$ associated to $V$ is
\[h_{m}(\Lambda) = -\int V(x)dM(x) = - \int V(x)e^{V(x)}dm(x).\]
If we call $P(x) = e^{V(x)}$, then,
\[h_{m}(\Lambda) = - \int P(x)\log(P(x))dm(x),\]
which is a classical expression of the entropy when there is no dynamics.

\medskip

The pressure of a measure $\nu$ satisfies
\[P_{m}(\nu) = \sup_{V\,normalized} [\int e^{V(x)}\,d\nu(x) - \int V(x)e^{V(x)}dm(x) ].\]
Then,
\[P_{m}(\nu) = \sup_{P>0,\,\int P(x)dm(x)=1} [\int P(x)\,d\nu(x) - \int P(x)\log P(x) dm(x) ].\]
If $d\nu = Udm$, then,
\[P_{m}(U\,m) = \sup_{P>0,\,\int P(x)dm(x)=1} [\int  U(x)\,P(x)dm(x) - \int \log (P(x))\,P(x)dm(x) ]\]
\[\stackrel{Prop. \,\ref{pressuretilde}}{=}  \log \int e^{U(x)}dm(x).\]

Note that $\int e^{U(x)}dm(x)$ (after normalization) is a classical expression for the Gibbs probability for the potential $U$ (when there is no dynamics).

\end{example}

\bigskip

\begin{example} \label{so}

Suppose that $[x]=\{x\}$ for any $x\in \Omega$, that is
$x\sim y$, if and only if, $x=y$. In this case any transverse function is a function $\nu$. Indeed,  for each $x$, we associate the class $[x]=\{x\}$, and then we assign to it a positive number $\nu(x)$. We fix as $\hat{\nu}^{y}$ the Dirac delta measure on $\{y\}$, for each $y\in \Omega$, that is, $\hat{\nu}$ is the constant function $1$. We consider fixed the Haar System $(G,\hat{\nu})$.

The unique Haar-normalized function is $V\equiv 0$ and any probability $M$ on $\Omega$ is Haar-invariant with Jacobian $e^{V}=1$. For any function $U$ we have $\log(\tilde U) = U$ and $U-\log(\tilde U) = 0 = V$.

For each probability $M$ we associate a transverse measure $\Lambda$ by
\[ \Lambda(\nu) = \iint e^{V(x)}\nu^{y}(dx)dM(y) = \int \nu(y) \, dM(y).\]
On the other hand, as $[x]=\{x\}$, the unique modular function is $\delta(x,x)=1$. Then, any transverse measure has the above form.

The entropy of any transverse probability $\Lambda$  is equal to
\[h_{\hat{\nu}}(\Lambda) = -\int V(x)dM(x) = 0,\]
and the pressure of a transverse function $\nu$ is given by
\[P_{\mu}(\nu) = \sup_{\Lambda} [\Lambda(\nu) +h_{\hat{\nu}}(\Lambda)] = \sup_{M\, probability}\int \nu(y) \, dM(y) =\sup_{y}\,\nu(y).\]

\end{example}

\end{document}